\documentclass{article}
\usepackage[margin=1in]{geometry}
\usepackage{enumitem}
\usepackage[french,english]{babel}
\usepackage{amsmath,amsfonts,amsthm,amssymb,bbm}
\usepackage{graphicx,color,dsfont}

\usepackage{hyperref}

\providecommand{\keywords}[1]{\textbf{\textit{Keywords:}} #1}

\newtheorem{theorem}{Theorem}
\newtheorem{lemma}{Lemma}

\newtheorem{corollary}{Corollary}

\definecolor{ao}{rgb}{0.0, 0.5, 0.0}
\definecolor{darkpastelgreen}{rgb}{0.01, 0.75, 0.24}
\definecolor{auburn}{rgb}{0.43, 0.21, 0.1}
\definecolor{armygreen}{rgb}{0.29, 0.33, 0.13}

 \newtheorem{thm}{Theorem}[section]
 
 \newtheorem{lem}[thm]{Lemma}
 
 \theoremstyle{definition}
 \newtheorem{defn}[thm]{Definition}
 \theoremstyle{remark}
 \newtheorem{rem}[thm]{Remark}
 
 \numberwithin{equation}{section}

\begin{document}

\title{Hardy inequality and fractional Leibnitz rule for perturbed Hamiltonians on the line }

\author{Vladimir Georgiev,
\\ Department of Mathematics, University of Pisa, Largo B. Pontecorvo 5, \\  Pisa,
56127 Italy, \\
 georgiev@dm.unipi.it \\
\and \\
 Anna Rita Giammetta, \\ Department of Mathematics, University of Pisa, Largo B. Pontecorvo 5, \\ Pisa,
56127 Italy, \\
giammetta@mail.dm.unipi.it}

%    General info

%\date{\today \, }

\maketitle

\selectlanguage{english}
\begin{abstract}
We consider the following perturbed Hamiltonian $\mathcal{H}= -\partial_x^2 + V(x)$ on the real line. The potential $V(x),$ satisfies a short range assumption of type  $$(1+|x|)^\gamma V(x) \in L^1(\mathds{R}), \   \gamma > 1.$$  We study the equivalence of classical
homogeneous  Sobolev type spaces  $\dot{H}^s_p(\mathds{R})$, $p \in (1,\infty)$ and the corresponding perturbed homogeneous Sobolev spaces associated with the perturbed Hamiltonian. It is shown that the assumption zero is not a resonance   guarantees that the perturbed and unperturbed homogeneous Sobolev norms of order $s = \gamma - 1 \in [0,1/p)$ are equivalent. As a corollary,  the corresponding wave operators leave classical homogeneous Sobolev spaces of order $s \in [0,1/p)$ invariant.
\end{abstract}

\selectlanguage{english}


\section{ Introduction and motivation}

The uncertainty principle in quantum mechanics  is frequently associated with Hardy type inequality
\begin{equation}\label{eq.II1}
   \| \,  |x|^{-s} f \|_{L^p(\mathds{R}^n)} \leq C \| \mathcal{H}_0^{s/2}f \|_{L^p(\mathds{R}^n)}, \ \ s \in [0, n/p),
\end{equation}
where $\mathcal{H}_0 = -\Delta$ is the free Hamiltonian in $\mathds{R}^n, n \geq 1.$
The presence of a perturbed Hamiltonian  $\mathcal{H}= \mathcal{H}_0 + V(x)$ with a short range real-valued  potential $V(x)$ leads to the natural question to verify if Hardy type inequality is true for this perturbed Hamiltonian. The appearance of eigenvectors of $\mathcal{H}$ is an obstacle to have Hardy type inequality or to establish existence and completeness of the wave operators in the whole $L^p(\mathds{R}^n)$ space, so   it is natural to look for estimate of type
\begin{equation}\label{eq.II2}
   \| \,  |x|^{-s} f \|_{L^p(\mathds{R}^n)} \leq C \| \mathcal{H}_{ac}^{s/2}f \|_{L^p(\mathds{R}^n)}, \ \ s \in [0, n/p),
\end{equation}
where $\mathcal{H}_{ac}$ is the absolutely continuous part of the perturbed Hamiltonian and $f$ is in the domain of $\mathcal{H}_{ac}$.

Our key goal in this work is to study the equivalence of the fractional energy norms 
\begin{equation}\label{eq.II3}
  \| \mathcal{H}_{ac}^{s/2}f \|_{L^p(\mathds{R})} \sim \| \mathcal{H}_{0}^{s/2}f \|_{L^p(\mathds{R})},
\end{equation}
since this equivalence property  shows that \eqref{eq.II1} implies \eqref{eq.II2}.

Another motivation to study the equivalence property  \eqref{eq.II3} is connected with the necessity to generalize so called fractional Leibnitz rule, used as a basic tool in rigorous analysis of local well-posedness of nonlinear dispersive equations, to the case of fractional Hamiltonians of type $\mathcal{H}_{ac}^{s/2}$.
To be more precise, the following estimate is  known as fractional Leibnitz rule or Kato-Ponce estimate (one can see \cite{bib:20} for the proof)
	\begin{alignat}{2}\label{eq:1.1}
	\| \mathcal{H}_0^{s/2}(fg)\|_{L^p(\mathds{R})}
	\leq C \| \mathcal{H}_0^{s/2} f \|_{L^{p_1}(\mathds{R})} \| g \|_{L^{p_2}(\mathds{R})} +
	C \| f \|_{L^{p_3}(\mathds{R})} \| \mathcal{H}_0^{s/2} g \|_{L^{p_4}(\mathds{R})},
	\end{alignat}
where the  parameters $s, p, p_j, j=1,\dots,4,$ satisfy
	\[
	s > 0,\ \
	1 < p, p_1,p_2,p_3,p_4 < \infty, \ \
	\frac{1}{p}= \frac{1}{p_1} + \frac{1}{p_2}
	= \frac{1}{p_3} + \frac{1}{p_4}.
	\]

The estimate can be considered as natural homogeneous version
of the non-homogeneous inequality of type \eqref{eq:1.1}
involving Bessel potentials $(1-\mathcal{H}_0)^{s/2}$ in the place of $\mathcal{H}_0^{s/2},$
obtained by Kato and Ponce in \cite{bib:23}
(for this the estimates of type \eqref{eq:1.1} are called Kato-Ponce estimates, too).
More general domain for parameters can be found in \cite{bib:17}.
A more precise estimate can be deduced when $0 < s < 1$.
More precisely,
Kenig, Ponce, and Vega \cite{bib:24} obtained the estimate
	\begin{align}
	\| \mathcal{H}_0^{s/2}(fg) - f \mathcal{H}_0^{s/2} g - g \mathcal{H}_0^{s/2} f\|_{L^p(\mathds{R})}
	\le C \| \mathcal{H}_0^{s_1/2} f \|_{L^{p_1}(\mathds{R})} \| \mathcal{H}_0^{s_2/2} g \|_{L^{p_2}(\mathds{R})},
	\label{eq:1.2}
	\end{align}
provided
	\[
	0 < s = s_1 + s_2 < 1, \ \ s_1, s_2 \ge 0,
	\]
and
	\begin{alignat}{2}\label{eq:1.3}
	1 < p, p_1,p_2 < \infty, \ \
	\frac{1}{p}= \frac{1}{p_1} +  \frac{1}{p_2} .
	\end{alignat}

Therefore, one can pose the question  to find appropriate short range assumptions on the perturbed Hamiltonian  so that the fractional Leibnitz rule \eqref{eq:1.1} or the more precise bilinear estimate \eqref{eq:1.2} are valid for this perturbed Hamiltonian.
Since the equivalence property \eqref{eq.II3} implies \eqref{eq:1.1}, it is important to determine admissible domain for the parameters $s >0, p \in (1,\infty),$
where  \eqref{eq.II3} holds. The uncertainty principle restriction $s < 1/p$ is a reasonable candidate and we aim at studying if this is the optimal domain where  \eqref{eq.II3} is fulfilled.

We can make another interpretation of  \eqref{eq.II3} connecting $\| \mathcal{H}_0^{s/2}f\|_{L^p(\mathds{R})}$ with  the homogeneous Sobolev  spaces $\dot{H}^s_p(\mathds{R})$ and observing that \eqref{eq.II3} guarantees the invariance of
the action of the wave operators
  $$ W_\pm = s - \lim_{t \to \pm \infty} P_{ac}(\mathcal{H})e^{it\mathcal{H}} e^{-it\mathcal{H}_0}$$
  on these homogeneous Sobolev spaces.

  The existence and completeness of the wave operators in  standard Hilbert space (typically Lebesgue space $L^2$) in case of short range perturbations is well known (see \cite{LPh64}, \cite{RSI78}, \cite{HII} and the references therein).
The  functional calculus for the absolutely continuous part $\mathcal{H}_{ac}= P_{ac}(\mathcal{H}) \mathcal{H} $ of the perturbed non-negative operator $\mathcal{H}$ can be introduced with a relation involving $W_\pm$
\begin{equation}\label{eq.III2}
   g(\mathcal{H}_{ac}) = W_+ g(\mathcal{H}_0)W_+^*= W_- g(\mathcal{H}_0)W_-^*,
\end{equation}
for any function $g \in L^\infty_{loc}(0,\infty).$ Moreover, the wave operators map unperturbed Sobolev spaces in the perturbed ones,
$$ W_\pm :  D(\mathcal{H}_0^{s/2}) \to  D(\mathcal{H}_{ac}^{s/2}) $$
and we have
$$ W_\pm : \dot{H}^s_p(\mathds{R})   \to  \dot{H}^s_{p,\mathcal{H}_{ac}}(\mathds{R}), \ \forall s \geq 0, \ 1 < p < \infty,$$
where $\dot{H}^s_{p,\mathcal{H}_{ac}}(\mathds{R})$ is the perturbed homogeneous Sobolev space generated by the Hamiltonian $\mathcal{H}_{ac}.$
More precisely, $\dot{H}^s_{p,\mathcal{H}_{ac}}(\mathds{R})$ is the  homogeneous Sobolev spaces associated with the absolutely continuous part $\mathcal{H}_{ac}$ of the  perturbed Hamiltonian $ \mathcal{H}= \mathcal{H}_0+V$.
This is
 the closure of  functions $f \in S(\mathds{R})$ orthogonal\footnote{ the precise definition of eigenvectors is given below in \eqref{V6a} } to the eigenvectors of $\mathcal{H}$ with respect to the norm
\begin{equation}\label{eq:BS2} \begin{aligned}
   \|  f\|_{ \dot{H}^s_{p,\mathcal{H}_{ac}}(\mathds{R})} =
\left\| \mathcal{H}^{s/2}_{ac}f\right\|_{L^p(\mathds{R})}.
\end{aligned}\end{equation}

The equivalence property \eqref{eq.II3}
implies that the homogeneous Sobolev space $\dot{H}^s_p(\mathds{R}) $ is  invariant under the action of the wave operators $W_\pm$
for $0 \leq s < 1/p$.

\section{Assumptions and main results}

%***** Begin GREEN COLOR

The study of the dispersive properties of the evolution flow in some cases of short range perturbed Hamiltonians $\mathcal{H}$ shows  (see \cite{CGV}, \cite{GV2003}) that  homogeneous Sobolev norms for perturbed and unperturbed Hamiltonians are equivalent
\begin{equation}\label{eq.MR0}
    \| \mathcal{H}_{ac}^{s/2}f \|_{L^2(\mathds{R}^n)}  \sim \|\mathcal{H}_0^{s/2}f\|_{L^2(\mathds{R}^n)},
\end{equation}
provided $s < n/2.$ Our goal is to extend this equivalence to the case
\begin{equation}\label{eq.MR0a}
    \| \mathcal{H}_{ac}^{s/2}f \|_{L^p(\mathds{R}^n)}  \sim \|\mathcal{H}_0^{s/2}f\|_{L^p(\mathds{R}^n)},
\end{equation}
with $s<n/p.$

First, we shall show that the requirement $s<n/p$ is optimal, i.e. we shall prove the following result:

\begin{theorem} \label{l.co1} If $n\geq 1$ and $V (x)$ is defined as follows
\begin{equation}\label{eq.C1a1}
   V(x) = \frac{1}{1+|x|^3},
\end{equation}
 then \eqref{eq.II2} with $s=n/p \leq 2$ is not true.
\end{theorem}

Our next goal is to obtain  \eqref{eq.II2} in the admissible range $s \in [0,n/p)$ for  the case $n=1.$ First we shall describe the assumptions on the potential $V.$

We shall assume that the potential  $V:\mathds{R}\to \mathds{R}$ is a real-valued potential, $V \in L^1(\mathds{R})$ and $V$ is decaying sufficiently rapidly at infinity, namely following \cite{W} we require
  \begin{equation}\label{V6}
    \|\langle x \rangle^\gamma V\|_{L^1(\mathds{R})}  < \infty, \  \ \gamma \geq 1,
\end{equation}
or equivalently we assume $ V\in L^1_{\gamma}(\mathds{R}) $, where
$$ \  L^1_\gamma(\mathds{R}) = \{f \in L^1_{loc}( \mathds{R}) ; \langle x \rangle^\gamma f(x) \in L^1(\mathds{R}) \}, \ \langle x \rangle^2 = 1+x^2. $$

Our key assumption on $V$ is that zero is not a resonance point. The precise definition of the notion of resonance point at the origin is given in Definition \ref{dres} by the aid of the relation $$T(0)=0.$$

The point spectrum of $\mathcal{H}$  consists of real numbers $\lambda \in (-\infty,0],$ such that
\begin{equation}\label{V6a}
   \mathcal{H} f - \lambda  f = 0, \ f \in L^2(\mathds{R}),
\end{equation}
and absolutely continuous part $[0,\infty).$ We shall denote by $L^2_{pp}(\mathds{R})$ the linear space generated by the eigenvectors $f$ in \eqref{V6a}. This is finite dimensional space and its orthogonal complement in $L^2$ is the invariant subspace, where the perturbed Hamiltonian $\mathcal{H}$ is absolutely continuous.

The key tool to prove the Hardy inequality and the fractional Leibnitz rule \eqref{eq:1.2} is the following estimate.
\begin{theorem} \label{MT1}
Suppose $$V\in L^1_\gamma(\mathds{R}), \  \gamma > 1, \  s=\gamma -1 < 1/p , \  p \in (1,\infty)$$ and  the perturbed Hamiltonian  $\mathcal{H}$ has no  resonance at the origin. Then there exists a positive constant $C=C(s,p)>0$ so that  we have
$$  \| (\mathcal{H}_{ac}^{s/2}  - \mathcal{H}_0^{s/2}) f \|_{ L^p(\mathds{R})} \leq C\|  f\|_{ L^q(\mathds{R})}, $$
for $1/p-1/q=s$ and $f \in S(\mathds{R})$.
\end{theorem}

It is natural to use a Paley-Littlewood localization associated with the perturbed Hamiltonian. Here and below
$\varphi(\tau) \in C_0^\infty(\mathds{R} \smallsetminus 0)$ is a non-negative  even function, such that
\begin{equation}\label{eq.PL1}
 \sum_{j \in \mathds{Z}} \varphi \left(
\frac{ \tau}{2^j} \right) = 1 \ , \ \ \forall \ \tau \in \mathds{R} \setminus  0
\end{equation}
and
\begin{equation}\label{eq.PL2}
\varphi\left( \frac{\tau}{2^k}  \right) \varphi\left( \frac{\tau}{2^\ell}  \right) = 0 , \ \forall \ k,\ell \in \mathds{Z} , \ |k-\ell| \geq 2.
\end{equation}
We set
\begin{equation}\label{eq.ES1}
\pi_k^{ac} = \varphi \left( \frac{\sqrt{\mathcal{H}_{ac}}}{2^k} \right), \ \pi^0_k = \varphi \left( \frac{\sqrt{\mathcal{H}_0}}{2^k}  \right).
\end{equation}

We have the following equivalent norm (see \cite{Ze10})
\begin{equation}\label{eq:MR1} \begin{aligned}
   \|  f\|_{ \dot{H}^s_{p,\mathcal{H}_{ac}}(\mathds{R})} \sim
\left\| \left( \sum_{k=-\infty}^\infty 2^{2ks} \left| \pi_k^{ac}f\right|^2 \right)^{1/2} \right\|_{L^p(\mathds{R})}.
\end{aligned}
\end{equation}

Our approach to prove Theorem \ref{MT1} is based on establishing estimate of the type.

\begin{lemma} \label{KL1}If the assumptions of Theorem \ref{MT1}  are fulfilled, then  for any $s \in (0,1/p)$ and $q \in (1,\infty)$ defined by
$$ \frac{1}{p} - \frac{1}{q} = s $$
we have
\begin{equation}\label{MR2}
\left\| \left\| 2^{ks} \left( \pi_k^{ac}-\pi_k^0\right)  f \right\|_{\ell^2_k}   \right\|_{L^p_x(\mathds{R})} \leq C \| f\|_{L^q(\mathds{R})}.
\end{equation}
\end{lemma}

Indeed if this estimate is verified, then we can use \eqref{eq:MR1} and see that \eqref{MR2} implies the assertion of Theorem \ref{MT1}.

Therefore, the estimate \eqref{MR2} is the key point in the proof of  Theorem \ref{MT1}.

\begin{corollary}
If the assumptions of Theorem \ref{MT1}  are fulfilled, then the equivalence property \eqref{eq.II3} holds.
\end{corollary}

\begin{proof}
 The results in \cite{DeiTru}, \cite{W0}, \cite{AY}, \cite{DF06}, \cite{Ze10} imply the existence and continuity of the wave operators in $L^p$, $1 < p < \infty,$ so
one can deduce Bernstein inequality
\begin{equation}\label{eq.MR3}
   \left\| \pi^{ac}_k f \right\|_{L^q(\mathds{R})} \leq C (2^k)^{1/p-1/q} \|f\|_{L^p(\mathds{R})}, \ \ 1 \leq p \leq q \leq \infty, \ k\in \mathbb{Z}
\end{equation}
and via the equivalence property \eqref{eq:MR1} we deduce  the Sobolev estimate
\begin{equation}\label{eq.MR4}
   \left\|  f \right\|_{L^q(\mathds{R})} \leq C  \|\mathcal{H}_{ac}^{s/2}f\|_{L^p(\mathds{R})}, \ \ 1 < p < q < \infty, \ \ s = \frac{1}{p}-\frac{1}{q}.
\end{equation}

From the estimate of Theorem \ref{MT1} now we can write
$$  \| (\mathcal{H}_{ac}^{s/2}  - \mathcal{H}_0^{s/2} )f \|_{ L^p(\mathds{R})} \leq C\|  f\|_{ L^q(\mathds{R})} \leq C  \|\mathcal{H}_{ac}^{s/2}f\|_{L^p(\mathds{R})} , $$
so we have
$$ \|  \mathcal{H}_0^{s/2} f \|_{ L^p(\mathds{R})} \leq C \|\mathcal{H}_{ac}^{s/2}f\|_{L^p(\mathds{R})}.$$
The opposite estimate can be deduced in the same way from Theorem \ref{MT1} and the "free" Sobolev estimate
\begin{equation}\label{eq.MR5}
   \left\|  f \right\|_{L^q(\mathds{R})} \leq C  \|\mathcal{H}_{0}^{s/2}f\|_{L^p(\mathds{R})}, \ \ 1 < p < q < \infty, \ \ s = \frac{1}{p}-\frac{1}{q}.
\end{equation}

This completes the proof.
\end{proof}

Theorem \ref{MT1} has also the following simple consequences.
\begin{corollary}
If the assumptions of Theorem \ref{MT1}  are fulfilled, then the Hardy inequality \eqref{eq.II2} holds.
\end{corollary}
\begin{corollary}
If the assumptions of Theorem \ref{MT1}  are fulfilled, then we have the  fractional Leibnitz rule, i.e.
\begin{align}
	\| \mathcal{H}_{ac}^{s/2}(fg) - f \mathcal{H}_{ac}^{s/2} g - g \mathcal{H}_{ac}^{s/2} f\|_{L^p(\mathds{R})}
	\le C \| \mathcal{H}_{ac}^{s_1/2} f \|_{L^{p_1}(\mathds{R})} \| \mathcal{H}_{ac}^{s_2/2} g \|_{L^{p_2}(\mathds{R})},
	\label{eq:1.2m}
	\end{align}
provided
\begin{alignat}{2}\label{eq:MR1.3}
	1 < p, p_1,p_2 < \infty, \ \
	\frac{1}{p}= \frac{1}{p_1} +  \frac{1}{p_2} 
	\end{alignat}
and 
	\[
	0 < s = s_1 + s_2 , \ \ s_1, s_2 \ge 0, s_1 < \frac{1}{p_1},\  s_2 < \frac{1}{p_2}.
	\]

\end{corollary}

Alternative application of the equivalence of the homogeneous Sobolev norms can be connected with the fractional power of the pseudo conformal generators, defined by 
\begin{equation}\label{eq.cc5}
   |J_0(t)|^s = t e^{{\rm i}x^2/(4t)} \mathcal{H}_0^{s/2} e^{-{\rm i}x^2/(4t)},  s \geq 0.
\end{equation}
These operators commute with  the free Schr\"odinger group $ e^{-{\rm i} \mathcal{H}_0 t},$ $ \mathcal{H}_0= -\partial_x^2.$ 

Natural  generalization of \eqref{eq.cc5} for the case of perturbed Schr\"odinger group $ e^{-{\rm i} \mathcal{H} t},$ $ \mathcal{H}= -\partial_x^2+V$ with short range potential is introduced in \cite{CGV} as follows
\begin{equation}\label{eq.cc6}
   |J(t)|^s = t e^{{\rm i}x^2/(4t)} \mathcal{H}^{s/2} e^{-{\rm i}x^2/(4t)}, \ s \geq 0.
\end{equation}
 In the case $V=0$ we have 
 $$ \|(t\partial_x - {\rm i} x/2)  \left( e^{-{\rm i} \mathcal{H}_0 t} f\right) \|_{L^2} \sim \||J_0(t)| \left( e^{-{\rm i} \mathcal{H}_0 t} f\right) \|_{L^2} $$
 so 
 the conservation of the pseudo conformal energy
 \begin{equation}\label{eq.cc3}
   \|(t\partial_x - {\rm i} x/2)  \left( e^{-{\rm i} \mathcal{H}_0 t} f\right) \|_{L^2} = \frac{1}{2}\|x   f \|_{L^2}
 \end{equation}
 and interpolation argument imply 
 \begin{equation}\label{eq.cc7}
   \||J_0(t)|^s \left( e^{-{\rm i} \mathcal{H}_0 t} f\right) \|_{L^2} \leq C \left(\| f \|_{H^1(\mathds{R}) }+ \|x   f \|_{L^2}\right), \ \forall t >0,
 \end{equation}
 for any $s \in [0,1].$
 
 One can use the the equivalence result as stated in Theorem \ref{MT1} and deduce (see Lemma 5.1 in \cite{CGV})
 \begin{equation}\label{eq.cc8}
   \||J_0(t)|^s g \|_{L^2} \sim \||J(t)|^s g \|_{L^2}
 \end{equation}
 for any $s \in [0,1/2).$

 We turn now to possible inflation phenomena manifested by the pseudo conformal norms over the perturbed Schr\"odinger flow, i.e. we shall study the quantity
 $$ \||J_0(t)|^s \left( e^{-{\rm i} \mathcal{H} t} f\right) \|_{L^2} $$
 when $s=1>1/2.$

\begin{lemma}
\label{theo:1}
Assume the potential $V \in L^\infty \cap L^1_\gamma(\mathds{R})$ with $\gamma > 1$  is such that
\begin{equation}\label{eq.co32}
    \int_{\mathds{R}} V(y) dy > 0 .
\end{equation}
Then for any initial data $f(x) \in S(\mathds{R})$ with
 \begin{equation}\label{eq.co33}
    f(0) \neq 0
 \end{equation}
 we have
\begin{equation}
\label{200}
\limsup_{t \to \infty} \ \|(t\partial_x - {\rm i} x/2)  \left( e^{-{\rm i} \mathcal{H} t} f\right) \|_{L^2}  = \infty  .
\end{equation}
\end{lemma}

\section{Idea to prove the key Lemma \ref{KL1}}

Our main tool to study the  kernel
$$  \varphi \left( \frac{\sqrt{\mathcal{H}_{ac}}}{M} \right)(x,y)$$
is the following representation of the kernel as filtered Fourier transform
\begin{equation}\label{eq.ft1}
\mathcal{F}_{\varphi,M} (a)(\xi) = \int \varphi \left( \frac{\tau}{M} \right) a(\tau) e^{-{\rm i} \xi \tau} d\tau
\end{equation}
of symbols $a(\tau)$ represented as linear combinations with constant coefficients  of functions in the set
\begin{equation}\label{eq.DO3m}
\mathcal{ A } = \left\{\ 1 , \ T(\tau), \ R_\pm(\tau)   \ \right\},
\end{equation}
or more generally
of symbols involving functions
$a(x,\tau)$ represented as linear combinations with constant coefficients  of functions in the set
\begin{equation}\label{eq.DO3}
\mathcal{ B } = \left\{   \widetilde{m_\pm}(x,\tau) ,\   T(\tau) \widetilde{m_\pm}(x,\tau),\  R_\pm(\tau) \widetilde{m_\pm}(x,\tau)  \ \right\},
\end{equation}
where $\widetilde{m_\pm}(x,\tau)  = m_\pm(x,\tau) -1,$ $m_\pm$ are modified Jost functions, while $T, R_\pm$  are the transmission and reflection  coefficients.

It is simple to establish that the kernel $\varphi(\sqrt{\mathcal{H}}/M)(x,y)$ can be decomposed as follows (one can see \cite{GG16}):
\begin{lem}\label{Ber1}
	If  $\varphi$ is an even non-negative function, such that  $\varphi \in C^\infty _0({\mathbf R} \setminus \{0\}),$ then for any $M>0$
	we have
	\begin{equation}\label{eq.Ber2a}
	\varphi \left( \frac{\sqrt{\mathcal{H}}}{M} \right)(x,y) =  K_M^0(x,y)+\widetilde{K}_M(x,y) ,
	\end{equation}
	where $K_M^0(x,y)$ can be represented as sum of the terms
	\begin{equation}\label{eq.lt1}
	\mathds{1}_{\epsilon_1 x>0}\mathds{1}_{\epsilon_2 y>0}\mathcal{F}_{\varphi,M}(a)(\epsilon_3 x +\epsilon_4 y)
	\end{equation}
	and the term $\widetilde{K}_M(x,y)$ is represented as sum of the terms
	\begin{gather}\label{eq.Ber2}
	\mathds{1}_{\epsilon_1 x>0}\mathds{1}_{\epsilon_2 y>0}\mathcal{F}_{\varphi,M}(b_1(x,\cdot))(\epsilon_3 x +\epsilon_4 y)+\mathds{1}_{\epsilon_1 x>0}\mathds{1}_{\epsilon_2 y>0}\mathcal{F}_{\varphi,M}(b_2(y,\cdot))(\epsilon_3 x
	+\epsilon_4 y)+\\
	+\mathds{1}_{\epsilon_1 x>0}\mathds{1}_{\epsilon_2 y>0}\mathcal{F}_{\varphi,M}(b_3(x,\cdot)b_4(y,\cdot))(\epsilon_3 x +\epsilon_4 y),\nonumber
	\end{gather}
	where $\epsilon_i=\pm 1$, for $i=1,\dots,4,$ $a(\tau)$ represents a linear combination with constant coefficients of functions in the set $\mathcal{A}$ in \eqref{eq.DO3m} and $b_i$, for $i=1,\dots,4$, are linear combinations with constant coefficients of functions in the set $\mathcal{B}$ in  \eqref{eq.DO3}.
\end{lem}

\begin{rem} We shall call the term $K_M^0(x,y)$ the leading one, with the following exact representation
	\begin{equation}\label{eq.lead}
	K_M^0(x,y)= c\int_{\mathds{R}} e^{-{\rm i} \tau(x - y)}\varphi\left(\frac{\tau}{M}\right) \alpha(x,y,\tau)\,d\tau
	\end{equation}
	with symmetric kernel $ \alpha(x,y,\tau) = \alpha(y,x,\tau)$ and
	\begin{equation*}
	\alpha(x,y,\tau)=
	\begin{cases}
	T(\tau) &x<0<y,\\
	(R_+(\tau)+1) e^{2i\tau x}-e^{2i\tau x}+1 & 0<x<y,\\
	(R_-(\tau)+1)e^{-2i\tau y}-e^{-2i\tau y}+1 & x<y<0.
	\end{cases}
	\end{equation*}	
	
	The term $\widetilde{K}_M(x,y)$ will be called the remainder one. In Lemma \ref{Ber1} to simplify the notation we neglected the symbolism $a^\pm$, $b_i^{\pm}$.
\end{rem}

A priori estimates for the remainder term are obtained using the estimates of the filtered Fourier transform established in Lemma \ref{lft4} and Lemma \ref{lft5}.
\begin{lem}\label{Ber1a}
	Suppose the condition \eqref{V6} is fulfilled with $\gamma \geq 1 +s,$ $ s \in (0,1),$  the operator $\mathcal{H}$ has no point spectrum and $0$ is not a resonance point  for $\mathcal{H}.$
	If  $\varphi$ is an even non-negative function, such that  $\varphi \in C^\infty _0({\mathbf R} \setminus \{0\}),$ then for any $p \in (1,1/s),$ any $M \in (0,\infty)$ and for any $b^\pm(x,\tau)$, $b_1^\pm(x,\tau)$, $b^\pm_2(x,\tau)$ in the set  \eqref{eq.DO3}
	we have
	\begin{gather}\label{eq.DO5}
	\left\| \int_{\mathds{R}}  \mathds{1}_{\pm x>0}\mathcal{F}_{\varphi,M}(b^\pm (x,\cdot))(x \pm  y) f(y)dy \right\|_{L_x^p(\mathds{R})} + \\ \nonumber +  \left\| \int_{\mathds{R}}  \mathds{1}_{\pm y>0}\mathcal{F}_{\varphi,M}(b^\pm (y,\cdot))(x \pm  y) f(y)dy \right\|_{L_x^p(\mathds{R})} \leq \frac{C}{\langle M \rangle} \|f\|_{L^q(\mathds{R})},
	\end{gather}
	and
	\begin{equation}\label{eq.DO6}
	\left\| \int_{\mathds{R}}  \mathds{1}_{\pm x>0}\mathds{1}_{\pm y>0}\mathcal{F}_{\varphi,M}(b_1^\pm(x,\cdot)b_2^\pm(y,\cdot))(x \pm  y) f(y)dy \right\|_{L_x^{p}(\mathds{R})}  \leq \frac{C}{\langle M \rangle } \|f\|_{L^q(\mathds{R})},
	\end{equation}
	where $\  \frac{1}{q}=\frac{1}{p}- s.$
\end{lem}

According with the notation introduced in \eqref{eq.ES1}, we set

\begin{equation}\label{eq.ES2}
\pi_{\leq k}^{ac} = \sum_{j \leq k} \pi_j^{ac}, \ \ \pi^{ac}_{\geq k} = \sum_{j \geq k} \pi^{ac}_j.
\end{equation}
\begin{equation*}
f_k= \pi_k^{ac} f, \ \ f_{\leq k}= \sum_{j\leq k}\pi_j^{ac} f, \ \  f_{\geq k }= \sum_{j\geq k}\pi_j^{ac} f, \ \  f_{k_1,k_2}= \sum_{k_1\leq j\leq k_2 }\pi_j^{ac} f
\end{equation*}
and respectively $f^0_k$, $f_{\leq k}^0$, $f_{\geq k}^0$, $f_{ k_1,k_2}^0$ defined as before replacing $\pi_j^{ac}$ with $\pi_j^0$.

Hence, the decomposition \eqref{eq.Ber2a} can be rewritten as follows
\begin{equation*}
\pi_k^{ac}=I_k- (\pi_k^{ac}-I_k),
\end{equation*}
where the operator $I_k$ represents the operators involved in the leading kernel and $(\pi_k^{ac}-I_k)$ is the remainder term.

To prove Lemma \ref{KL1} we will establish the following inequalities:

	\begin{equation}\label{6.3}
	\left\| \left\| 2^{ks} \left( \pi_k^{ac} - I_k\right)  f \right\|_{\ell^2_k}   \right\|_{L^p_x(\mathds{R})} \leq C \| f\|_{L^q(\mathds{R})},
	\end{equation}
	\begin{equation}\label{6.3b}
	\left\| \left\| 2^{ks} \left(  I_k-\pi_k^0 \right)  f \right\|_{\ell^2_k}   \right\|_{L^p_x(\mathds{R})} \leq C \| f\|_{L^q(\mathds{R})},
	\end{equation}
	with $1/p=1/q+s$ and
	$I_k$ are the operators
	$$ I_k(f)(x) = \int_\mathds{R} K_{2^k}^0(x,y) f(y) dy $$
	with kernels representing
	the leading term \eqref{eq.lt1} in the expansion of Lemma \ref{Ber1} of $\pi_k.$

\section{Sup and H\"older type arpiori estimates}
\label{sec:spectral}

\subsection{Estimates for the modified Jost functions}

In this section we recall some classical results concerning the spectral decomposition of the perturbed Hamiltonian.
Recall that the  Jost functions are  solutions $f_{
	\pm } (x,\tau )=e^{\pm i\tau x}m_{
	\pm } (x,\tau )$ of $\mathcal{H }u=\tau ^2 u$ with
$$ \lim _{x\to +\infty }   {m_{ + } (x,\tau )}  =1 =
\lim _{x\to -\infty }  {m_{- } (x,\tau )}  . $$
We set  $x_+:=\max \{ 0,x \}$,   $x_-:=\max \{ 0,-x \}$.

%***** Begin GREEN COLOR

The  estimate and  the asymptotic expansions of $m_\pm(x,\tau)$ are based on  the following integral equations
\begin{alignat}{2}\label{eq.Igra1n}
m_\pm(x,\tau) = 1+ K_{\pm}^{(\tau)} (m_{\pm}(\cdot, \tau))(x),
\end{alignat}
where
$ K_\pm^{(\tau)}$ is the integral  operator defined as follows
$$ K_\pm^{(\tau)}  (f)(x) =
\pm  \int_x^{\pm \infty} D(\pm(t-x),\tau) V(t) f(t) dt $$
and
\begin{equation}\label{eq.AE3}
D(t,\tau) = \frac{e^{2it\tau}-1}{2i\tau} = \int_0^t e^{2iy\tau} dy;
\end{equation}

The following lemma  is well known.
\begin{lem} (see Lemma 1 p. 130 \cite{DeiTru} and Lemma 2.1 in \cite{W})
	\label{lem:Jost} Assume  $V \in L^{1}_\gamma(\mathds{R})$, $ \gamma \in (1,2].$ Then we have the properties:
	\begin{enumerate}[noitemsep,label=\alph*)]
		\item for any $x \in \mathds{R}$ the function
		\begin{equation}\label{eq.Jo1}
		\tau \in \overline{\mathds{C}_\pm} \mapsto m_\pm (x , \tau),\ \ \mathds{C}_\pm = \{\tau \in \mathds{C}; {\rm Im} \tau \gtrless 0 \}
		\end{equation}
		is analytic in $\mathds{C}_\pm$ and $  C^{1} (\overline{\mathds{C}_\pm}  );$
		\item  there exist  constants $C_1$  and $C_2>0$
		such that for any $x, \tau \in \mathds{R}$:
		\begin{align}  \label{eq:kernel2n}
		&  \mathds{1}_{\pm x >0}|m_\pm(x, \tau )-1|\le  C _1
		\langle \tau \rangle ^{-1}  \ ;
		\\&   \label{eq:tderm}   \mathds{1}_{\pm x >0}|  \partial _\tau   m_\pm(x, \tau ) | \le  \frac{C_2}{|\tau|^{\gamma-2}\langle \tau \rangle^{\gamma-1}} . \end{align}
	\end{enumerate}
\end{lem}
A slight improvement is given in the next Lemma.
\begin{lem}\label{lem:Jostk0} ( see \cite{GG16})
	Suppose $V\in L^1_\gamma(\mathds{R})$ with $\gamma\geq 1$. Then we have the following properties:
	\begin{itemize}
		\item[a)] There exists a constant $C>0$ such that for any $x\in \mathds{R}$, $\tau\in \overline{\mathbb{C}_\pm}$, we have
		\begin{equation}\label{eq.1k0}
		\left|m_\pm(x,\tau)-1\right|\leq C\frac{\langle x_\mp\rangle}{\langle x_\pm\rangle^{\gamma-1}};
		\end{equation}
		\item[b)]There exists a constant $C>0$ such that for any $x\in \mathds{R}$, $\tau\in \overline{\mathbb{C}_\pm}\smallsetminus\{0\}$, we have
		\begin{equation}\label{eq.2k0}
		\left|m_\pm(x,\tau)-1\right|\leq C\frac{\langle x_\mp\rangle}{\langle x_\pm\rangle^{\gamma}|\tau|};
		\end{equation}
		\item[c)]Let $\sigma\in[0,1)$. Then there exists a constant $C>0$ such that for any $x\in \mathds{R}$ we have
		\begin{equation}\label{eq.3k0}
		\left\|m_\pm(x,\tau)-1\right\|_{C^{0,\sigma}(\mathbb{C}_\pm)}\leq C\frac{\langle x_\mp\rangle^{1+\sigma}}{\langle x_\pm\rangle^{\gamma-1-\sigma}}, \ \gamma>1,\  0\leq\sigma\leq \gamma-1;
		\end{equation}
		\item[d)]Let $\sigma\in[0,1)$. Then there exists a constant $C>0$ such that for any $x\in \mathds{R}$ we have
		\begin{equation}\label{eq.4k0}
		\left\|\tau (m_\pm(x,\tau)-1)\right\|_{C^{0,\sigma}(\mathbb{C}_\pm)}\leq C\frac{\langle x_\mp\rangle^{1+\sigma}}{\langle x_\pm\rangle^{\gamma-\sigma}}, \ \gamma>1.
		\end{equation}	
	\end{itemize}
\end{lem}

\subsection{Estimates  for transmision and reflection coefficients}

The transmission coefficient
$T(\tau )$ and the reflection  coefficients
$R_\pm (\tau )$   are defined by  the formula
\begin{equation}  \label{eq:kernel35} \begin{aligned} &    T(\tau )m_\mp (x ,\tau )= R_\pm (\tau )e^{\pm 2{\rm i} \tau x }m_\pm (x,\tau )+   m_\pm (x,-\tau ).
\end{aligned}
\end{equation}
From \cite{DeiTru}  and from \cite{W} we have the following lemma.
\begin{lem}
	\label{lem:TRcoeff} We have the following properties of the transmissions and reflection coefficients.
	\begin{enumerate}[noitemsep,label=\alph*)]
		\item $T, R_\pm   \in C (\mathds{R}   )$.
		\item There exists $C_1,C_2>0$  such that:
		\begin{align} &  \label{eq:TRcoeff0}
		|T(  \tau )-1| +| R_\pm (  \tau )  |\le C_1 \langle \tau  \rangle ^{ -1} \\&  %\text{ for  $C=C(\|   V    \|  _{L^{1,1}} )$;}
		\label{eq:TRcoeff}
		|T(\tau )|^2+ |R_{\pm }(\tau )|^2=1.
		\end{align}
		\item If $T(0)=0,$ (i.e. zero is not a resonance point), then for some $\alpha \in \mathds{C} \setminus \{0\}$ and for some
		$\alpha_+, \alpha_- \in \mathds{C}$
		\begin{alignat}{2}\label{eq.TRa}
		& T(\tau) = \alpha \tau + o(\tau), \ \ 1+R_\pm(\tau) = \alpha_\pm \tau + o(\tau) \ \ \mbox{as $\tau \to 0$}, \\ \nonumber
		T(\tau) = & 1 + O(|\tau|^{-1}), \ \ R_\pm(\tau) = O(|\tau|^{-1}) \ \ \mbox{as $\tau \to \infty$}.
		\end{alignat}
		\item  there exists a  constant $C>0$  such that for any $\tau \in \mathds{R}$:
		\begin{align}  \label{eq:Tra1}
		&  T^\prime(\tau) \le  C \langle \tau \rangle ^{-1}  .
		\end{align}
	\end{enumerate}
\end{lem}

The property  c) in the last Lemma suggests the following.
\begin{defn} \label{dres} The origin is a resonance point for the hamiltonian $\mathcal{H}$ if and only if
	$$T(0) \neq 0.$$
\end{defn}

Therefore, taking a bump function $\varphi \in C_0^\infty((0,\infty))$ (with support in  $[1/2,2]$ for example), we have estimates in the the algebra $C([0,4]) $
of the terms of type
\begin{equation}\label{eq.mas101}
\left\| \varphi(\cdot) T(M\cdot ) \right\|_{C^{0}([0,4])}+ \left\| \varphi(\cdot) \left( R_\pm(M\cdot)+1 \right) \right\|_{C^{0}([0,4])} \leq C M^{}
\end{equation}
and
\begin{equation}\label{eq.mas102}
\left\| \frac{\varphi(\cdot)}{ T(M\cdot )} \right\|_{C^{0}([0,4])}+ \left\| \frac{\varphi(\cdot)}{ \left( R_\pm(M\cdot)+1 \right)} \right\|_{C^{0}([0,4])} \leq C M^{-1}
\end{equation}
for $M \in (0,1]$.

We can use the  assumption $V \in L^1_\gamma(\mathds{R}),$ $\gamma > 1,$ to get some more precise H\"older type bounds.

\begin{lem}
	\label{lem:TCH1} Suppose $V \in L^1_\gamma(\mathds{R})$ with $\gamma > 1$  and $T(0)=0.$ Then  for  any
	$\sigma \in (0,s]$ and $M \in (0,1]$ we have:
	\begin{enumerate}[noitemsep,label=\alph*)]
		\item $T, R_\pm   \in C^{0,\sigma} (\mathds{R}   );$
		\item for $M \in (0,1)$  we have
		\begin{equation}\label{eq.mas1}
		\left\| \varphi(\cdot) T(M\cdot ) \right\|_{C^{0,\sigma}((0,+\infty))}+ \left\|  \varphi(\cdot) \left( R_\pm(M\cdot)+1 \right) \right\|_{C^{0,\sigma}((1/2,2))} \leq C M^{};
		\end{equation}
		\item for $M \in [1,\infty)$  we have
		\begin{equation}\label{eq.mas2}
		\left\| \varphi(\cdot) \left( T(M\cdot )-1 \right) \right\|_{C^{0,\sigma}((0,+\infty))}+ \left\|  \varphi(\cdot) R_\pm(M\cdot) \right\|_{C^{0,\sigma}((1/2,2))} \leq C M^{-1}.
		\end{equation}
	\end{enumerate}
\end{lem}
\begin{proof}
	
	The proof is based on the  relations
	\begin{equation}\label{eq.63a}
	\frac{\tau}{T(\tau)} = \tau - \frac{1}{2i} \int_\mathds{R} V(t) m_+(t,\tau) dt, \ \  \tau \in  \mathds{R }\setminus \{0\},
	\end{equation}
	\begin{equation}\label{eq.TRC7}
	R_\pm (\tau) = \frac{T(\tau)}{2i\tau} \int_\mathds{R} e^{\mp 2it\tau} V(t) m_\mp(t,\tau) dt , \  \tau \in  \mathds{R} \setminus \{0\}
	\end{equation}
	and the properties of the functions $m_\mp(t,\tau)$ from Lemma \ref{lem:Jostk0}. Indeed, we can get the estimates
	\begin{equation}\label{eq.mas103}
	\left\| \frac{\varphi(\cdot)}{ T(M\cdot )} \right\|_{C^{0,\sigma}([0,4])}+ \left\| \frac{\varphi(\cdot)}{ \left( R_\pm(M\cdot)+1 \right)} \right\|_{C^{0,\sigma}([0,4])} \leq C M^{-1}
	\end{equation}
	first. Further,  we can use the fact\footnote{the problem to have norm-controlled inversion in smooth Banach algebra is well-known and some more general results and references can be found in \cite{GK13}  } that we can control the norm of the inverse of $ f$ in the subalgebra $C^{0,\sigma}$ by
	the norm of $f$ in $C^{0,\sigma}$ and the norm of $1/f$ in $C(T)$
	$$ \left\| \frac{\varphi(\cdot)}{ f(\cdot )} \right\|_{C^{0,\sigma}([0,4])} \leq C \left\|\frac{\widetilde{\varphi}(\cdot)}{ f(\cdot )}  \right\|_{C^{0}([0,4])}  +   \frac{\left\|\widetilde{\varphi}(\cdot) f \right \|_{C^{0,\sigma}([0,4])}}{ \|f(\cdot )\|_{C^0([0,4])}^2 } ,$$
	where $\widetilde{\varphi} \in C_0^\infty((0,\infty))$ has slightly larger support in  $[1/2-\delta,2+\delta]$ with $\delta >0$ sufficiently small.
	Applying this estimate and the estimate
	\eqref{eq.mas102} and \eqref{eq.mas103} with $\varphi $ replaced by a cut-off function with slightly larger support, we complete the proof.
\end{proof}

\section{Estimates of the filtered Fourier transform of $m_\pm -1$}
Given a bump  function $\varphi \in C_0^\infty(\mathds{R}),$  we define the corresponding filtered Fourier transform as in \eqref{eq.ft1}.
We shall distinguish two different cases.
If the bump function $\varphi \in C_0^\infty((0,\infty))$ is such that \eqref{eq.PL1} and \eqref{eq.PL2} are satisfied, then we can assert that
$\varphi(\tau/M) $ has a support with $\tau \sim M$.

The integral equation \eqref{eq.Igra1n} with sign $+$ can be rewritten as
\begin{alignat}{2}\label{eq.ft2}
\widetilde{ m_+}(x,\tau) = &\int_x^\infty \int_0^{t-x} e^{2{\rm i} \tau y} V(t) dy dt
+ \int_x^\infty \int_0^{t-x} e^{2{\rm i} \tau y} V(t)\widetilde{m_+}(t,\tau) dy dt,
\end{alignat}
where
$$ \widetilde{m_+}(x,\tau) = m_+(x,\tau) -1.$$
If we assume that $V\in L^1_{\gamma}(\mathds{R})$, $\gamma=1+s,$
then the assertion of Lemma \ref{lem:Jostk0} guarantees that $\widetilde{m_+}(x,\tau)=m_+(x,\tau)-1$ is in $L^1_{x>0}(\mathds{R}).$

Applying the filtered Fourier transform and setting
$$ g_M(\xi;x) = \int_{\mathds{R}} e^{-{\rm i} \tau \xi} \widetilde{m_+}(x,\tau)  \varphi\left( \frac{\tau}{M}\right) d\tau = \mathcal{F}_{\varphi,M} (\widetilde{m_+}(x,\cdot)) (\xi ),$$
we get
\begin{alignat}{2}\label{eq.ft3}
g_M(\xi;x) = &\underbrace{M \int_x^\infty \int_0^{t-x}  V(t)  \widehat{\varphi}(M(\xi-2y)) dy dt}_{a_M(\xi;x)} + \\ \nonumber
+  & \int_x^\infty \int_0^{t-x}  V(t)g_M(\xi-2y;t) dy dt.
\end{alignat}

We have the following pointwise estimates.

%{\color{red}nei successivi risultati non vengono usati da alcuna parte i risultati della sezione 3. giusto? }SI
%{\color{red}dove vengono usati? }nella prova dell'equivalenza
\begin{lem}\label{lft2}
	If $\varphi \in C_0^\infty(\mathds{R}),$ satisfies \eqref{eq.PL1}, \eqref{eq.PL2} and $V \in L^1_\gamma(\mathds{R}),$ $\gamma =1+s$, $s \in (0,1),$
	then for $M \in (0,1)$ the filtered Fourier transform
	$$ \mathcal{F}_{\varphi,M} \left(\widetilde{m_\pm}(x,\cdot)\right) (\xi ) =  \int_{\mathds{R}} e^{-{\rm i} \tau \xi}  \widetilde{m_\pm}(x,\tau) \varphi\left(\frac{\tau}{M}\right) d\tau $$
	satisfies the pointwise estimates:
	\begin{itemize}
		\item one can find functions
		$$ F_M^\pm (\xi) \in L^1(\mathds{R}), \ \ \|F_M^\pm\|_{L^1(\mathds{R})} \leq C(\|V\|_{L^1_{1+s}(\mathds{R})}) \|\widehat{\varphi}\|_{L^1(\mathds{R})},$$
		so that
		\begin{equation}\label{eq.ot1}
		\mathds{1}_{\{\pm x >0\}} \langle x \rangle^s \left| \mathcal{F}_{\varphi,M} (\widetilde{m_\pm}(x,\cdot)) (\xi ) \right|  \leq  F_M^\pm(\xi).
		\end{equation}
	\end{itemize}
\end{lem}

\begin{proof}
	We choose the sign $+$ in \eqref{eq.ot1} for determinacy.
	To prove \eqref{eq.ot1} we set
	$$ G_M(\xi;x) = \mathds{1}_{\{ x >0\}} \sup_{\eta < \xi} |g_M(\eta; x)|\langle x \rangle^s,$$
	where $ g_M(\xi; x)$ is the Filtered Fourier transform of the remainder $ \widetilde{m_+}(x,\tau) = m_+(x,\tau) -1,$ satisfying the integral equation \eqref{eq.ft3}.
	The function
	\begin{equation}\label{eq.ft5}
	F_M(\xi) = M \int_0^\infty \langle t \rangle^\gamma |V(t)| \int_0^t |\widehat{\varphi} (M(\xi - 2y))| dy dt ,
	\end{equation}
	satisfies
	\begin{equation}\label{eq.ft5a}
	F_M(\xi) \in L^1(\mathds{R}), \ \ \|F_M\|_{L^1(\mathds{R})} \leq \|V\|_{L^1_\gamma(\mathds{R})} \|\widehat{\varphi}\|_{L^1(\mathds{R})}.
	\end{equation}
	Moreover, since we are considering the case $x>0$ we get easily the following estimates
	\begin{equation*}
	|\mathds{1}_{x>0}\langle x \rangle^s a_M(\xi;x)|\leq F_M(\xi),
	\end{equation*}
	where $a_M(\xi;x)$ is defined in \eqref{eq.ft3}.
	Hence, coming back to $G_M(\xi;x)$ and recalling \eqref{eq.ft3} we have
	\begin{alignat}{2}\label{eq.ft6}
	G_M(\xi;x) \leq &F_M(\xi)  +  \int_x^\infty   \langle t \rangle|V(t)|  G_M(\xi;t)  dt, \ \forall x >0.
	\end{alignat}
	Applying the Gronwall lemma we get
	$$ G_M(\xi;x) \leq  C F_M(\xi),$$
	where $C$ is a positive constant depending on $\|V\|_{L^1_1(\mathds{R})}$ and $F_M(\xi)$ satisfies \eqref{eq.ft5} and \eqref{eq.ft5a}. This completes the proof.
\end{proof}

If $M\geq 1$ and $\varphi $ satisfying \eqref{eq.PL1} and \eqref{eq.PL2}, then we can improve the results of Lemma \ref{lft2}.
Indeed, the term $a_M(\xi;x)$ in \eqref{eq.ft3} can be rewritten as follows
\begin{equation*}
a_M(\xi;x)= M\int_{x}^{\infty}dt\int_{\mathds{R}}d\tau V(t)e^{-i\tau M \xi}\varphi(\tau)\frac{e^{2iM\tau(x-y)}-1}{2iM\tau}.
\end{equation*}
Hence we have that
\begin{equation*}
|\mathds{1}_{x>0} \langle x \rangle^sa_M(\xi;x)|\leq F^{(1)}_M(\xi),
\end{equation*}
where
\begin{equation}
F^{(1)}_M(\xi)= \int_x^\infty \langle t\rangle^s|V(t)||\hat{\varphi}(M\xi)|\,dt
\end{equation}
and
\begin{equation*}
\|F^{(1)}_M(\xi)\|_{L^1(\mathds{R})}\leq \frac{1}{M}\|V\|_{L^1_s(\mathds{R})} \|\hat{\varphi}\|_{L^1(\mathds{R})}.
\end{equation*}

Proceeding as in the proof of Lemma \ref{lft2} we get the following result.
\begin{lem}\label{lft2m1}
	If $\varphi $ satisfies \eqref{eq.PL1} and \eqref{eq.PL2} and $V \in L^1_\gamma(\mathds{R}), $ $\gamma =1+s$, $s \in (0,1),$
	then for $M \in (0,\infty)$ the filtered Fourier transform
	$$ \mathcal{F}_{\varphi,M} (\widetilde{m_\pm}(x,\cdot)) (\xi ) =  \int_{\mathds{R}} e^{-{\rm i} \tau \xi} \left( \widetilde{m_\pm}(x,\tau) \right) \varphi\left(\frac{\tau}{M}\right) d\tau $$
	satisfies the pointwise estimates:
	\begin{itemize}
		\item one can find functions
		$$ F_M^\pm (\xi) \in L^1(\mathds{R}), \ \ \|F_M^\pm\|_{L^1(\mathds{R})} \leq\frac{1}{\langle M\rangle} C(\|V\|_{L^1_{1+s}(\mathds{R})}) \|\widehat{\varphi}\|_{L^1(\mathds{R})},$$
		so that
		\begin{equation}\label{eq.ot1m1}
		\mathds{1}_{\{\pm x >0\}} \langle x \rangle^{s} \left| \mathcal{F}_{\varphi,M} (\widetilde{m_\pm}(x,\cdot)) (\xi ) \right|  \leq  F_M^\pm(\xi).
		\end{equation}
	\end{itemize}
	
\end{lem}

One can use a Wiener type argument and deduce estimates for $T(\tau), R_\pm(\tau)+1.$
\begin{lem}\label{lft3} (see \cite{DF06}, \cite{Ze10})
	If $\varphi \in C_0^\infty(\mathds{R})$ obeys \eqref{eq.PL1}, \eqref{eq.PL2} and $V \in L^1_\gamma(\mathds{R}),$ $\gamma =1+s$, $s \in (0,1),$
	then for  $M \in (0,\infty)$ the filtered Fourier transforms
	$$ \mathcal{F}_{\varphi,M} (T(\cdot)) (\xi ) =  \int_{\mathds{R}} e^{-{\rm i} \tau \xi} T(\tau) \varphi\left(\frac{\tau}{M}\right) d\tau $$
	and
	$$ \mathcal{F}_{\varphi,M} (R_\pm (\cdot)+1) (\xi ) =  \int_{\mathds{R}} e^{-{\rm i} \tau \xi} (R_\pm (\tau)+1) \varphi\left(\frac{\tau}{M}\right) d\tau $$
	are in $L^1(\mathds{R})$ and the following inequality are satisfied
	\begin{gather*}\label{eq.FFf4}
	\|\mathcal{F}_{\varphi,M}( T(\cdot)) (\xi ) \|_{L^1(\mathds{R})} + \|\mathcal{F}_{\varphi,M} (R_\pm(\cdot)+1) (\xi ) \|_{L^1(\mathds{R})} \leq C(\|V\|_{L^1_{1+s}(\mathds{R})}) \|\widehat{\varphi}\|_{L^1(\mathds{R})}, \ M\in(0,1),\\
	\|\mathcal{F}_{\varphi,M}( T(\cdot)-1) (\xi ) \|_{L^1(\mathds{R})} + \|\mathcal{F}_{\varphi,M} R_\pm(\cdot) (\xi ) \|_{L^1(\mathds{R})} \leq \frac{1}{\langle M\rangle}C(\|V\|_{L^1_{1+s}(\mathds{R})}) \|\widehat{\varphi}\|_{L^1(\mathds{R})}, \ M>1.
	\end{gather*}
\end{lem}

Turning to the estimates \eqref{eq.ot1}, we see that
$$   a(x,\xi) =  \mathds{1}_{\{\pm x >0\}}\mathcal{F}_{\varphi,M} (\widetilde{m_\pm}(x,\cdot)) (\xi ) $$
satisfies estimate
\begin{equation}\label{eq.ot20}
|a(x,\xi)| \leq a_1(x) a_2(\xi), \ a_1 \in L^{1/s, \infty}(\mathds{R}), \  a_2 \in L^1(\mathds{R}),
\end{equation}
where $a_1(x)=\langle x \rangle^{-s}$.
Lemma \ref{lft3} guarantees that
$$ b(\xi) =   \mathcal{F}_{\varphi,M} (T(\cdot)) (\xi ) \in L^1(\mathds{R}).$$
Since
$$ \mathds{1}_{\{\pm x >0\}}\mathcal{F}_{\varphi,M} (T(\cdot) (\widetilde{m_\pm}(x,\cdot))) (\xi ) =  a(x, \cdot) * b(\cdot)(\xi),$$
we see that
$$ \left| a(x, \cdot) * b(\cdot)(\xi) \right| \leq a_1(x) \underbrace{a_2* |b|}_{\widetilde{a_2}}(\xi), \ a_1 \in L^{1/s, \infty}(\mathds{R}),  \ \widetilde{a_2 }\in L^1(\mathds{R}),$$
since
$$ L^1 * L^1 \subset L^1$$
due to the Young inequality.

The above inclusion actually can be modified in a way suitable for our a priori estimates as follows
\begin{equation}\label{eq.DO1}
\left( L^1 \cap L^\infty \right)  * \left( L^1 \cap L^\infty \right) \subset  \left( L^1 \cap L^\infty \right).
\end{equation}

This observation leads to the following.

\begin{lem}\label{lft4}
	If $\varphi \in C_0^\infty(\mathds{R}),$  $V \in L^1_\gamma(\mathds{R}), $ $\gamma =1+s$, $s \in (0,1),$ and
	$a^\pm (x,\tau)$ is any function in the set
	\begin{equation}\label{eq.ot24}
	\left\{ \widetilde{m_\pm}(x,\tau) ,\  T(\tau) \widetilde{m_\pm}(x,\tau),\  (R_\pm(\tau)+1) \widetilde{m_\pm}(x,\tau)   \right\},
	\end{equation}
	then for $M \in (0,\infty)$ the filtered Fourier transform
	$$ \mathcal{F}_{\varphi,M} (a^\pm (x,\cdot)) (\xi ) =  \int_{\mathds{R}} e^{-{\rm i} \tau \xi} a^\pm (x,\tau) \varphi\left(\frac{\tau}{M}\right) d\tau $$
	satisfies the pointwise estimates:
	\begin{equation}\label{eq.ot25}
	\mathds{1}_{\{\pm x >0\}}\left| \mathcal{F}_{\varphi,M} (a^\pm (x,\cdot)) (\xi ) \right|  \leq  f_1(x) f^{(M)}_2(\xi),
	\end{equation}
	where
	\begin{equation*}
	f_1 (x) \in L^{1/s,\infty}(\mathds{R}) \cap L^\infty(\mathds{R}) , \ f^{(M)}_2(\xi) \in L^1(\mathds{R})
	\end{equation*}
	and $\|f^{(M)}_2\|_{L^1(\mathds{R})} \leq C/\langle M\rangle $.
\end{lem}

Finally we consider products of type $a^\pm (x,\tau) b^\pm (y,\tau),$ where $a,b$ are in the set \eqref{eq.ot24} and we have the following estimates.

\begin{lem}\label{lft5}
	If $\varphi \in C_0^\infty(\mathds{R})$ is a bump function satisfying \eqref{eq.PL1}, \eqref{eq.PL2}, $V \in L^1_\gamma(\mathds{R}), $ $\gamma =1+s$, $s \in (0,1),$
	then for $M \in (0,\infty)$ the filtered Fourier transform of $a^\pm (x,\tau)b^\pm (y,\tau)$
	satisfies the pointwise estimate:
	\begin{equation}\label{eq.ot27}
	\mathds{1}_{\pm x>0}\mathds{1}_{\pm y >0} \left| \mathcal{F}_{\varphi,M} (a^\pm (x,\cdot)b^\pm (y,\cdot)) (\xi ) \right|  \leq  f_1(x) f^{(M)}_2(\xi) f_3(y),
	\end{equation}
	where
	$$  f_1 , f_3 \in L^{1/s,\infty}(\mathds{R}) \cap L^\infty(\mathds{R}) , \ f_2^{(M)}(\xi) \in L^1(\mathds{R}), \ \|f_2^{(M)}\|_{L^1(\mathds{R})} \leq \frac{C}{\langle M\rangle} $$
	with some constant $C>0$ independent of $M.$
\end{lem}

Now we can proceed with the proof of Lemma \ref{Ber1}.
% vedi da qui

\begin{proof}[Proof of Lemma \ref{Ber1}] To fix the idea and to simplify the notation we consider the case involving $b^+(y,\tau)=b(y,\tau)$. We separate two cases: $M \in (0,1]$ and $M \geq 1.$ For $M \in (0,1]$ our first step is to prove
	\begin{alignat}{2}\label{eq.DO7}
	& \left\| \int_{\mathds{R}}  \mathds{1}_{y>0}\mathcal{F}_{\varphi,M}(b(y,\cdot))(x \pm  y) f(y)dy \right\|_{L_x^p(\mathds{R})}  \leq C \|f\|_{L^q(\mathds{R})}.
	\end{alignat}
	We use the pointwise estimate \eqref{eq.ot25} so we can write
	$$\mathds{1}_{y>0}\left| \mathcal{F}_{\varphi,M}(b(y,\cdot))(x \pm  y) \right| \leq B_1^{(M)}(x\pm y) B_2(y) ,$$
	where $$ B_1^{(M)}\in L^1(\mathds{R}),  \ \  \|B_1^{(M)}\|_{L^1(\mathds{R})} \leq C , \ \ B_2 \in L^{1/s,\infty}(\mathds{R})  $$
	and \eqref{eq.DO7} follows from Young inequality
	\begin{equation}\label{eq.DO8}
	\left\| B_1^{(M)} * (B_2 f) \right\|_{L_x^p(\mathds{R})} \leq C \|B_1^{(M)}\|_{L^{1}(\mathds{R})}\|B_2 f\|_{L^p(\mathds{R})},
	\end{equation}
	and the H\"older estimate
	\begin{equation}\label{eq.DO9}
	\left\| B_2 f  \right\|_{L^p(\mathds{R})} \leq C \|f\|_{L^{q}(\mathds{R})}, \ \  B_2 \in L^{1/s,\infty}(\mathds{R}), \  \frac{1}q = \frac{1}p -s.
	\end{equation}
	Similarly, to prove
	\begin{alignat}{2}\label{eq.DO11}
	& \left\| \int_{\mathds{R}}  \mathds{1}_{x>0}\mathcal{F}_{\varphi,M}(b(x,\cdot))(x \pm  y) f(y)dy \right\|_{L_x^p(\mathds{R})}  \leq C \|f\|_{L^q(\mathds{R})}
	\end{alignat}
	we use the pointwise estimate \eqref{eq.ot25} again, so we can write
	$$\mathds{1}_{x>0}\left| \mathcal{F}_{\varphi,M}(b(x,\cdot))(x \pm  y) \right| \leq B_1^{(M)}(x\pm y) B_2(x) ,$$
	where $$ B_1^{(M)}\in L^1(\mathds{R}),  \ \  \|B_1^{(M)}\|_{L^1(\mathds{R})} \leq C , \  B_2 \in L^{1/s,\infty}(\mathds{R}).  $$
	This time we have to estimate the term
	$$   \left\| B_2( B_1^{(M)} *  f) \right\|_{L_x^p(\mathds{R})}$$
	so first we apply H\"older estimate \eqref{eq.DO9} and then the Young convolution inequality.
	
	Finally, the estimate \eqref{eq.DO6} follows from \eqref{eq.ot27} since we have
	$$ \mathds{1}_{x>0}\mathds{1}_{y>0}\left| \mathcal{F}_{\varphi,M}(b_1(x,\cdot)b_2(y,\cdot))(x \pm  y) \right| \leq  B_1^{(M)}(x\pm y) B_2(y) B_3(x) ,$$
	where
	$$ B_1^{(M)} \in L^1(\mathds{R}), \ \ \|B_1^{(M)}\|_{L^1(\mathds{R})} \leq C, \ \ B_2(y),  B_3(x)\in  L^{1/s,\infty}(\mathds{R}) \cap L^\infty(\mathds{R}) .$$
	This completes the proof for the case $M \in (0,1].$ For $M \geq 1$ we simply use the fact that we have better estimate
	$$\|B_1^{(M)}\|_{L^1(\mathds{R})} \leq C M^{-1}$$ and we prove \eqref{eq.DO5} and \eqref{eq.DO6} assuming $V \in L^1_1(\mathds{R})$ only. This completes the proof.
\end{proof}

\section{ Equivalence of homogeneous Sobolev norms }
In this section we are going to prove Lemma \ref{KL1}.

\begin{proof}[Proof of the inequality \eqref{6.3}]
	The relation \eqref{eq.Ber2} guarantees that
	$$ \pi_k^{ac}( f)(x)  -  I_k(f)(x)  $$
	can be represented as a sum of remainder terms of the form
	\begin{gather*}
	\sum_{\epsilon_1,\dots,\epsilon_4=\pm 1}\mathds{1}_{\epsilon_1 x>0}\int_{\mathds{R}}\mathds{1}_{\epsilon_2 y>0}\mathcal{F}_{\varphi,M}(b_1(x,\cdot))(\epsilon_3 x +\epsilon_4 y)f(y)\,dy+\\
	+\sum_{\epsilon_1,\dots,\epsilon_4=\pm 1 } \mathds{1}_{\epsilon_1 x>0}\int_{\mathds{R}}\mathds{1}_{\epsilon_2 y>0}\mathcal{F}_{\varphi,M}(b_2(y,\cdot))(\epsilon_3 x
	+\epsilon_4 y)f(y)\,dy+\nonumber\\
	+\sum_{\epsilon_1,\dots,\epsilon_4=\pm 1} \mathds{1}_{\epsilon_1 x>0}\int_{\mathds{R}}\mathds{1}_{\epsilon_2 y>0}\mathcal{F}_{\varphi,M}(b_3(x,\cdot)b_4(y,\cdot))(\epsilon_3 x +\epsilon_4 y)f(y)\,dy,\nonumber
	\end{gather*}
	such that the estimates of Lemma \ref{Ber1} imply
	$$ \left\| \left(\pi_k - I_k  \right)f \right\|_{L^p(\mathds{R})} \leq \frac{C}{\langle 2^k \rangle } \|f\|_{L^q(\mathds{R})},$$
	with
	$$ \frac{1}{q}   =\frac{1}{p}-s.$$
	Using the inequalities
	$$ \left\| \left\| 2^{ks} \left( \pi_k - I_k\right)  f\right\|_{\ell^2_k}   \right\|_{L^p_x(\mathds{R})}  \leq
	\left\| \left\| 2^{ks} \left( \pi_k - I_k\right)  f \right\|_{\ell^1_k}   \right\|_{L^p_x(\mathds{R})} \leq  $$
	$$ \leq  \left\| \left\| 2^{ks} \left( \pi_k - I_k\right)  f  \right\|_{L^p_x(\mathds{R})} \right\|_{\ell^1_k} \leq
	\left\| \frac{2^{ks}}{\langle 2^k \rangle} \right\|_{\ell^1_k}\left\| f \right\|_{L^q_x(\mathds{R})},$$
	and so we deduce \eqref{6.3}.
	
	This completes the proof.
\end{proof}

\begin{proof}[Proof of Lemma \ref{KL1}]
Our main goal is to establish the following estimate
\begin{equation}\label{eq.1}
\left\|\left\|2^{ks}(\pi_k-\pi^0_k)f \right\|_{\ell^2_{k}}\right\|_{L^p(\mathds{R})}\leq C \|f\|_{L^q(\mathds{R})},
\end{equation}
with $1/q=1/p-s$.

We start proving that
\begin{equation}\label{eq.1a}
\left\|\left\|2^{ks}(\pi_k-\pi^0_k)f \right\|_{\ell^2_{k\leq 0}}\right\|_{L^p(\mathds{R})}\leq C \|f\|_{L^q(\mathds{R})}.
\end{equation}

In particular, it will be enough to prove the inequality \eqref{6.3b}, i.e.
\begin{equation*}
\left\|\left\|2^{ks}(I_k-\pi^0_k)f \right\|_{\ell^2_{k\leq 0}}\right\|_{L^p(\mathds{R})}\leq C \|f\|_{L^q(\mathds{R})},
\end{equation*}
since the estimate \eqref{6.3} has been just established above.

Using the decomposition
\begin{equation*}
f=\sum_{j\in\mathbb{Z}} f_j^0,
\end{equation*}
we have that
\begin{equation}\label{eq.3}
\left( I_k-\pi_k^0\right)f = \left( I_k-\pi_k^0\right) f^0_{k-2,k+2}.
\end{equation}
Indeed, if follows from
\begin{equation*}
\left( I_k-\pi_k^0 \right) f^0_{\leq {k-2}}(x)= \int \int e^{i(x+y)\tau}\varphi\left(\frac{\tau}{2^k}\right) f^0_{\leq k-2}(y)\,d\tau\,dy =0
\end{equation*}
and
\begin{equation*}
\left( I_k-\pi_k^0 \right) f^0_{\geq {k-2}}(x)= \int \int e^{i(x+y)\tau}\varphi\left(\frac{\tau}{2^k}\right) f^0_{\geq k-2}(y)\,d\tau\,dy =0.
\end{equation*}
Moreover, the expression of the leading term shows that the kernel $\left( I_k-\pi_k^0 \right)(x,y)$
can be represented as sum of the terms
\begin{equation*}
\mathds{1}_{\epsilon_1 x>0}\mathds{1}_{\epsilon_2 y>0}\mathcal{F}_{\varphi,M}(a)(\epsilon_3 x +\epsilon_4 y), \ \
\end{equation*}
with $\epsilon_j=\pm 1, j=1,\dots,4,$ $\epsilon_1 \epsilon_2 \epsilon_3 \epsilon_4 =1$ and $a \in \mathcal{A},$ defined in \eqref{eq.DO3m}.

For simplicity we consider the case $a=1,$ $\epsilon_j=1,$ $\forall j=1,\dots,4,$ and we shall estimate the term
\begin{equation*}
\int\mathds{1}_{x>0}\mathds{1}_{y>0} e^{i\tau(x+y)}\varphi\left(\frac{\tau}{M}\right)\,d\tau.
\end{equation*}
Then, we can proceed similarly for the other terms.

Integrating by parts and using Lemma \ref{lem:Jostk0}, we get
\begin{align*}
\left\|2^{ks}\int \int \mathds{1}_{x>0}\mathds{1}_{y>0} e^{i\tau(x+y)}\varphi\left(\frac{\tau}{2^k}\right)f^0_k(y)\,d\tau\,dy\right\|_{\ell^2_{k\leq 0}}\leq \\
\leq C\int  \left\|\frac{2^{k(s+1)}\mathds{1}_{x>0}\mathds{1}_{y>0}}{\langle 2^k (x+y)\rangle^{1+s} }f_k^0(y)\,dy\right\|_{\ell^2_{k\leq 0}}\,dy\\
\leq C \int \left\|\frac{2^{k(s+1)} \mathds{1}_{x>0}\mathds{1}_{y>0}}{\langle 2^k (x+y) \rangle ^{1+s} }\right\|_{\ell^\infty_{k\leq 0}} \left\|f_k^0\right\|_{\ell^2_{k\leq 0}}\,dy.
\end{align*}
From the trivial inequality
\begin{equation*}
\left\|\frac{2^{k(s+1)}}{\langle 2^k x \rangle^{1+s}}\right\|_{\ell^\infty_{k\leq 0}}\leq \frac{C}{|x|^{1+s}}
\end{equation*}
combined with the Young inequality in Lorentz spaces we have
\begin{equation*}
\left\|\left\|2^{ks}\int \int \mathds{1}_{x>0}\mathds{1}_{y>0} e^{i\tau(x+y)}\varphi\left(\frac{\tau}{2^k}\right)f^0_k(y)\,d\tau\,dy\right\|_{\ell^2_{k\leq 0}}\right\|_{L^p(\mathds{R})}\leq
C \left\| \left\| f_k^0(y)\right\|_{\ell^2_{k\leq 0}}\right\|_{L^q(\mathds{R})},
\end{equation*}
with $1/q=1/p-s$ and $0<s<1/p$.

The case $k\geq 0$ follows similarly using the estimate
\begin{equation*}
\left|(\pi_k-\pi_k^0)f(x)\right|\leq C \int \frac{f(y)}{\langle 2^k (x\pm y) \rangle^s}\left(\frac{1}{\langle x\rangle}+\frac{1}{\langle y \rangle}\right)\,dy.
\end{equation*}
This complete the proof.
\end{proof}

\section{Counterexample for equivalence of homogeneous Sobolev spaces}

In this section we consider the   case $p \in [n/2,\infty) \cap (1,\infty)$ and we shall prove Theorem \ref{l.co1}, therefore we shall show that the equivalence property
\begin{equation}\label{eq.c1}
\|(\mathcal{H}_0+V)^{n/(2p)} u \|_{L^p(\mathds{R}^n)} \sim  \|(\mathcal{H}_0)^{n/(2p)}  u \|_{L^p(\mathds{R}^n)}
\end{equation}
is not true for $n \in \mathbb{N}.$
\begin{proof}[Proof of Theorem \ref{l.co1}]
	Let us suppose that the relation \eqref{eq.c1} holds. Choosing positive potential
	$$ V(x) = \frac{1}{1+|x|^3},$$
	we can apply the heat kernel estimate obtained in \cite{Z00}, i.e.
	\begin{equation}\label{eq.CE1}
	\frac{C_1 e^{-c_1|x-y|^2/4t}}{t^{n/2}} \leq   e^{-t \mathcal{H}}(x,y) \leq \frac{C_2 e^{-c_2|x-y|^2/4t}}{t^{n/2}}.
	\end{equation}
	This estimate and the relation
	$$ \mathcal{H}^{-\alpha} = \frac{1}{\Gamma(\alpha)} \ \int_0^\infty t^{\alpha-1} e^{-t \mathcal{H}}   dt $$
	imply
	$$\left|(\mathcal{H}_0+V)^{-1} u (x) \right| \leq C  \left|(\mathcal{H}_0)^{-1} u (x) \right|$$
	so taking the $L^p$ norm and using a duality argument,  we can write
	\begin{equation}\label{eq.CE2}
	\|V(\mathcal{H}_0+V)^{-1} f \|_{L^p(\mathds{R}^n)} \leq C \|f\|_{L^p(\mathds{R}^n)},
	\end{equation}
	so we have
	\begin{equation}\label{eq.CE3}
	\|V g\|_{L^p(\mathds{R}^n)} \leq C \|(\mathcal{H}_0+V)g\|_{L^p(\mathds{R}^n)}.
	\end{equation}
	Interpolation argument and the assumption $p \geq n/2$ combined with  the equivalence property \eqref{eq.c1} lead to
	\begin{equation}\label{eq.co2}
	\int_{\mathds{R}^n} (V(x))^{n/2} |u(x)|^p dx \leq C \|\mathcal{H}_0^{n/(2p)}u\|_{L^p(\mathds{R}^n)}^{p}.
	\end{equation}
	Taking $u $ in the Schwartz class $ S(\mathds{R}^n) $
	of rapidly decreasing function, we can apply a rescaling argument. Indeed, considering the dilation
	$$ u_\lambda(x) =  u(x\lambda),$$
	we find
	$$   \|\mathcal{H}_0^{n/(2p)}u_\lambda\|_{L^p(\mathds{R}^n)}^p =  \underbrace{\|\mathcal{H}_0^{n/(2p)} u\|_{L^2(\mathds{R}^n)}^{p}}_{\mbox{constant in $\lambda$}} $$
	and
	$$ \lim_{\lambda \searrow 0}\int_{\mathds{R}^n} V^{n/2}(x) |u_\lambda(x)|^p dx = \left(\int_{\mathds{R}^n} V^{n/2}(x)  dx\right) |u(0)|^p. $$ In this way we deduce
	\begin{equation}\label{eq.co2a1}
	|u(0)|^p \left(\int_{\mathds{R}^n} V^{n/2}(x)  dx\right) \leq  C \|\mathcal{H}_0^{n/(2p)} u\|_{L^p(\mathds{R}^n)}^p.
	\end{equation}
	The homogeneous norm
	$$ \|\mathcal{H}_0^{n/(2p)} u\|_{L^p(\mathds{R}^n)}^p$$ is also invariant under translations, i.e. setting
	$$ u^{(\tau)}(x) = u(x+\tau), $$ we have
	$$ \widehat{u^{(\tau)}}(\xi) = e^{{\rm i} \tau \xi} \widehat{u}(\xi) $$ and
	$$  \|\mathcal{H}_0^{n/(2p)}u^{(\tau)}\|_{L^p(\mathds{R}^n)}^p =\|\mathcal{H}_0^{n/(2p)} u\|_{L^p(\mathds{R}^n)}^p,$$ so applying \eqref{eq.co2a1} with $u^{(\tau)}$ in the place of  $u$, we find
	$$
	|u(\tau)|^p \int_{\mathds{R}^n} V^{n/2}(x)  dx \leq  C \|\mathcal{H}_0^{n/(2p)}u\|_{L^p(\mathds{R}^n)}^p,
	$$
	or equivalently
	\begin{equation}\label{eq.co2a2}
	\|u\|^p_{L^\infty(\mathds{R}^n)} \leq C_1 \|\mathcal{H}_0^{n/(2p)}u\|_{L^p(\mathds{R}^n)}^p,
	\end{equation}
	where
	$$ C_1 = \frac{C}{\|V^{n/2}\|_{L^1(\mathds{R}^n)}}.$$
	
	The substitution  $ \phi = \mathcal{H}_0^{n/(2p)} u $ enables us to rewrite \eqref{eq.co2a2} as
	\begin{equation}\label{eq.co2a3}
	\|I_{n/p}(\phi)\|^p_{L^\infty(\mathds{R}^n)} \leq C_1 \|\phi\|_{L^p(\mathds{R}^n)}^p,
	\end{equation}
	where
	$$ I_\alpha(\phi)(x) = \mathcal{H}_0^{-\alpha/2}(\phi)(x) = c \int_{\mathds{R}^n} |x-y|^{-n+\alpha} \phi(y) dy, \  \alpha \in (0,n) $$
	are the Riesz operators.
	
	It is easy to show that \eqref{eq.co2a3} leads to a contradiction. Indeed, taking
	$$ \phi_N(x) = \sum_{j=0}^N \underbrace{|x|^{-n/p}\  \mathds{1}_{2^j \leq |x| \leq  2^{j+1}} (x)}_{\chi_j(x)}, $$
	with $N \geq 2$ sufficiently large and being $\mathds{1}_A(x)$ the characteristic function of the set $A$.
	Since the functions $\chi_j$ have almost disjoint supports and they are non-negative, for almost every $x \in \mathds{R}$ we have
	$$ \sum_{j=1}^N \chi_j^p(x) = \left(\sum_{j=1}^N \chi_j(x) \right)^p . $$
	so
	$$  \|\phi_N\|_{L^p(\mathds{R}^n)}^p = \sum_{j=0}^N \int_{2^j}^{2^{j+1}} \frac{r^{n-1}dr}{r^n}  \leq C' N.$$
	Further, we can use the estimates
	$$ I_{n/p}(\phi_N)(0) \geq  \left( \sum_{j=0}^N \int_{2^j}^{2^{j+1}} \frac{r^{n-1}dr}{r^n} \right) \geq C N.$$
	Hence, from \eqref{eq.co2a3} we deduce
	$$ C N^p \leq  \|I_{n/p}(\phi_N)\|^p_{L^\infty(\mathds{R}^n)} \leq C_1 \|\phi_N\|_{L^p(\mathds{R}^n)}^p \leq C_2 N ,$$
	for any $N$ sufficiently big and this is impossible.
	
	This completes the proof of the Theorem.
\end{proof}

\section{Proof of Lemma \ref{theo:1}}

\begin{proof}[Step I: Pseudo conformal two parameter group $U(T,S)$]
Set
  \begin{equation}\label{eq.In7}
    \psi(t) =  e^{-{\rm i} (t-1) \mathcal{H}} f,   t>1 ,
\end{equation}
where $\mathcal{H}_0=-\partial_x^2.$

Making the  transformation
$$ (t,\psi) \ \Longrightarrow (T,\Psi), $$such that
where
\begin{equation}\label{eq:In8}
  t=\frac{1}{T}, \ \ \Psi(T,x)= \overline{\psi\left(\frac{1}{T},x \right)}.
\end{equation}
We can rewrite \eqref{eq.In7} as follows
\begin{equation}\label{eq.In7a}
    \Psi(T) =  e^{{\rm i} ( \mathcal{H}/T - \mathcal{H} )} \overline{f} .
\end{equation}
Now we can use the isometry
$$ B(T) : L^2(\mathds{R}) \ \to \ L^2(\mathds{R}),$$
associated with the pseudo conformal transform for the free Schr\"odinger equation, i.e.
\begin{equation}\label{eq.In7b}
    B(T) = M(T) \sigma_{T},
\end{equation}
with
\begin{equation}\label{eq.In7c}
   M(T)g(x) = e^{{\rm i} x^2/(4T)} g(x) , \ \ \sigma_T(g)(x) = T^{-1/2} g(T^{-1}x).
\end{equation}
Making the substitution
$$\Phi(T) = B(T) \Psi(T),$$
 we find  the integral equation
\begin{equation}\label{eq.In11}
    \Phi(T) = U(T) U^*(1) \Phi_0, \ \ \Phi_0 = B(1) \left( \overline{f}\right),
\end{equation}
where
\begin{equation}\label{eq.In7a}
 U(T) = B(T) e^{{\rm i} \mathcal{H}/T }
\end{equation}

We shall need the following properties of the two parameter group
$$U(T,S) = U(T)U^*(S) = B(T) e^{{\rm i} ( \mathcal{H}/T-\mathcal{H}/S ) } B^*(S).$$
\begin{lem}
If
\begin{equation}\label{eq.co40}
  -\Delta(T) =   -\Delta + T^{-2} V \left(\frac{x}{T} \right),
\end{equation}
then for any $T \in (0,1]$ this operator is self - adjoint positive, we have the group property
\begin{equation}\label{eq.co41}
    U(T_1,T_2)U(T_2,T_3) = U(T_1,T_3), \ \ \forall T_1,T_2,T_3 \in (0,1]
\end{equation}
and for any couple $T,S \in (0,1]$ we have
\begin{equation}\label{eq.co45}
    U(T,S) : D((-\Delta(S))^{a/2}) \ \to \\ D((-\Delta(T))^{a/2}), \ \ \forall a \in [0,2].
\end{equation}
\end{lem}

Note that we have the relation
\begin{equation}\label{eq.co46}
    \|(t\partial_x + {\rm i} x) \psi(t) \|_{L^2} \sim \|(-\Delta)^{1/2} \Phi(T) \|_{L^2}, \ \ T=1/t.
\end{equation}

 Hence the proof of Lemma \ref{theo:1} is reduced to the proof of  the following estimate.
 \begin{lem} \label{t.mes} For any $f\in S(\mathds{R})$ with  $f(0) \neq 0$ we have
 $$ \limsup_{T \searrow 0}  \| \Phi(T) \|_{H^1(\mathds{R})} = \infty  .
$$
\end{lem}

\end{proof}

\begin{proof}[Step II: Proof of Lemma \ref{t.mes}]

We shall argue by contradiction. If the assertion of the Theorem is not true then we can find $C>0$ so that
\begin{equation}\label{eq.co8}
    \| \Phi(T) \|_{H^1(\mathds{R})} \leq C \|\Phi_0\|_{H^1(\mathds{R})} , \ \forall T \in (0,1]  .
\end{equation}

The two parameter group $U(T,S)$ has the property
\begin{equation}\label{eq.co5}
    U(T,S) : D((-\Delta(S))^{a/2}) \ \to \\ D((-\Delta(T))^{a/2}), \ \ \forall a \in [0,2].
\end{equation}
and this means that we have in particular the inequality
\begin{equation}\label{eq.co6}
    \| (-\Delta(T))^{1/2}\underbrace{U(T,1) \Phi_0}_{\Phi(T)}\|_{L^2(\mathds{R})} \leq C \| (1-\Delta(1))^{1/2}\Phi_0\|_{L^2(\mathds{R})} \leq C \|\Phi_0\|_{H^1(\mathds{R})},
\end{equation}
since we assume $V \in L^\infty(\mathds{R}).$
The property \eqref{eq.co8} implies now
\begin{equation}\label{eq.co6}
    \| (-\Delta)^{1/2}\underbrace{U(T,1) \Phi_0}_{\Phi(T)}\|_{L^2(\mathds{R})} \leq C \|\Phi_0\|_{H^1(\mathds{R})},
\end{equation}
so using the relation
\begin{equation}\label{eq.co9}
  \| (-\Delta(T))^{1/2} \Phi(T)\|^2_{L^2(\mathds{R})} = \| (-\Delta)^{1/2} \Phi(T)\|^2_{L^2(\mathds{R})} + T^{-2} \int V(x/T) |\Phi(T,x)|^2 dx ,
\end{equation}
we get
\begin{equation}\label{eq.co6}
    T^{-2} \int V(x/T) |\Phi(T,x)|^2 dx \leq C \|\Phi_0\|^2_{H^1(\mathds{R})},
\end{equation}

This is equivalent to the relation
\begin{equation}\label{eq.co6}
    \int V(y) |\phi(T,yT)|^2 dy \leq C  T \|\Phi_0\|^2_{H^1(\mathds{R})},
\end{equation}
so using the assumption
$\int V(y) dy \neq 0$ 
and taking the limit $T \to 0,$ we get
$$ \Phi_0(0)=f(0)=0.$$
This is a contradiction and the proof of the Lemma is complete.

\end{proof}

\section{Modified Lax pairs relations}

Given any   two different perturbed  groups $ U(T,S),$ $ \widetilde{U}(T,S)$ connected via the splitting relation
\begin{equation}\label{eq.hi1a}
    U(T,S) = B(T) \widetilde{U}(T,S) B^*(S), \ 0 < T,S \leq 1
\end{equation}
 the corresponding time dependent generators $-{\rm i} H(T)$ and $-{\rm i} \widetilde{H}(T)$ are determined by the Cauchy problems
\begin{equation}\label{eq.hi1e}
   \frac{d}{dT} U(T,S) = - {\rm i} H(T) U(T,S), \ \ U(S,S) = I.
\end{equation}
\begin{equation}\label{eq.hi1f}
   \frac{d}{dT} \widetilde{U}(T,S) = - {\rm i} \widetilde{H}(T) \widetilde{U}(T,S), \ \ \widetilde{U}(S,S) = I.
\end{equation}
Now \eqref{eq.hi1a} can be associated with the following Lax pairs relation
\begin{equation}\label{eq.hi1g}
   B^\prime(T) = {\rm i} \left[ B(T) \widetilde{H}(T) - H(T) B(T) \right]
\end{equation}
and we can easily see that \eqref{eq.hi1g} implies that $-{\rm i} H(T)$ is the generator of the perturbed group $U(T,S).$

Now we apply this argument for
\begin{equation}\label{eq.hi1h}
    U_0(T,S) = B(T) \widetilde{U}_0(T,S) B^*(S), \ 0 < T,S \leq 1
\end{equation}
with
$$ U_0(T,S) =  e^{-{\rm i} \mathcal{H}_0(T -S) } =U_0(T) U_0^*(S), \ \ \widetilde{U}_0(T,S) = e^{{\rm i} \mathcal{H}_0/T} e^{-{\rm i} \mathcal{H}_0/S} . $$
Obviously, the generator of $U_0(T,S)$ is $-{\rm i} \mathcal{H}_0 = {\rm i} \partial_x^2$ and the Lax pairs relation
becomes now
\begin{equation}\label{eq.re5}
    B^\prime(T) = {\rm i} \left[ B(T)\frac{\mathcal{H}_0}{ T^2} - \mathcal{H}_0 B(T) \right]
\end{equation}
The check of this relation is straightforward and we omit it.

Now we can define the family of operators
\begin{equation}\label{eq.In7ab}
 U_0(T) = B(T) e^{{\rm i} \mathcal{H}_0/T }
\end{equation}

This relation and the definition of $B(T)$ imply
\begin{equation}\label{eq.In7a1}
 U_0(T) = e^{-{\rm i} \mathcal{H}_0 T}.
\end{equation}

Further  the perturbed group $U(T,S)$ defined by \eqref{eq.In7a} is of the form introduced in \eqref{eq.hi1a} with
$$ \widetilde{U}(T,S)= e^{{\rm i} \mathcal{H}/T } e^{-{\rm i} \mathcal{H}/S }  $$
and obviously the  generator of $\widetilde{U}(T,S)$  is $-{\rm i} \widetilde{H}(T) = - {\rm i} \mathcal{H}/T^2.$
The modified Lax pairs relation has the form
\begin{equation}\label{eq.re5as}
    B^\prime(T) = {\rm i} \left[ B(T)\frac{\mathcal{H}}{ T^2} - H(T) B(T) \right]
\end{equation}
and this relation is true with
$$ H(T) = -\Delta(T),$$
where
\begin{equation}\label{eq.CR9}
  -\Delta(T) =  T^{-2} \sigma_T  \mathcal{H }\sigma_T^* = -\Delta + T^{-2} V \left(\frac{x}{T} \right).
\end{equation}
Again the check of the relation is trivial consequence of \eqref{eq.re5} and we omit the details.

\section*{\normalsize{ Acknowledgements}}
The authors are grateful to Atanas Stefanov for the critical remarks and discussions during the preparation of the work.

Funding: This work was supported by by University of Pisa, project no. PRA-2016-41 "Fenomeni singolari in problemi deterministici
e stocastici ed applicazioni"; by INDAM, GNAMPA - Gruppo Nazionale per l'Analisi Matematica, la
Probabilit\`{a} e le loro Applicazioni and by Institute of Mathematics
and Informatics, Bulgarian Academy of Sciences.

\bibliographystyle{amsplain}

\end{document}